\documentclass[11pt]{amsart}
\usepackage{times}
\usepackage{tikz-cd}

\newcommand{\EE}{\mathbb{E}}

\newcommand{\GG}{\mathbb{G}}
\newcommand{\NN}{\mathbb{N}}

\newcommand{\ZZ}{\mathbb{Z}}

\newcommand{\PP}{\mathbb{P}}

\newcommand{\cB}{\mathcal{B}}

\newcommand{\cG}{\mathcal{G}}

\newcommand{\cM}{\mathcal{M}}

\newcommand{\cX}{\mathcal{X}}

\newcommand{\fa}{\mathfrak{a}}

\newcommand{\fb}{\mathfrak{b}}

\newcommand{\fe}{\mathfrak{e}}

\newcommand{\fg}{\mathfrak{g}}

\newcommand{\fK}{\mathfrak{K}}
\newcommand{\KK}{\mathbbm{k}}

\newcommand{\fm}{\mathfrak{m}}

\newcommand{\fn}{\mathfrak{n}}

\newcommand{\fs}{\mathfrak{s}}
\newcommand{\fsl}{\mathfrak{sl}}

\newcommand{\ft}{\mathfrak{t}}

\newcommand{\fu}{\mathfrak{u}}
\newcommand{\fv}{\mathfrak{v}}

\newcommand{\lra}{\longrightarrow}

\DeclareMathOperator{\ad}{ad}

\DeclareMathOperator{\Char}{char}

\DeclareMathOperator{\CJT}{\mathsf{CJT}}

\DeclareMathOperator{\msdeg}{\mathsf{deg}}

\DeclareMathOperator{\Gr}{Gr}

\DeclareMathOperator{\id}{id}
\DeclareMathOperator{\Jt}{Jt}

\DeclareMathOperator{\Lie}{Lie}
\DeclareMathOperator{\pr}{pr}
\DeclareMathOperator{\modd}{mod}

\DeclareMathOperator{\Rad}{Rad}
\DeclareMathOperator{\Max}{Max}
\DeclareMathOperator{\Nor}{Nor}
\DeclareMathOperator{\rk}{rk}

\usepackage{pictex}
\usepackage{amscd}
\usepackage{latexsym}
\usepackage{amssymb}
\usepackage{eucal}
\usepackage{eufrak}
\usepackage{verbatim}
\usepackage{bbm}

\numberwithin{equation}{section}

\newtheorem{Theorem}{Theorem}[section]
\newtheorem{Lemma}[Theorem]{Lemma}
\newtheorem{Corollary}[Theorem]{Corollary}
\newtheorem{Proposition}[Theorem]{Proposition}

\theoremstyle{Theorem}

\newtheorem*{thm*}{Theorem A}
\newtheorem*{thm**}{Theorem B}
\theoremstyle{remark}

\newtheorem*{Remarks}{Remarks}
\newtheorem*{Definition}{Definition}

\numberwithin{equation}{section}

\begin{document}

\title[Varieties and Endotrivial Modules]{Varieties of subalgebras and endotrivial modules}

\author[H. Chang and R. Farnsteiner]{Hao Chang${}^\ast$ \lowercase{and} Rolf Farnsteiner}
\address[H. Chang]{School of Mathematics and Statistics, and Key Laboratory of Nonlinear Analysis \& Applications (Ministry of Education), Central China Normal University, Wuhan 430079, P. R. China}
\email{chang@ccnu.edu.cn}
\thanks{${}^\ast$ Supported by the National Natural Science Foundation of China (No. 11801204) and the Fundamental Research Funds for the Central Universities (No. CCNU22QN002).}
\address[R. Farnsteiner]{Mathematisches Seminar, Christian-Albrechts-Universit\"at zu Kiel, Heinrich-Hecht-Platz 6, 24118 Kiel, Germany}
\email{rolf@math.uni-kiel.de}


\makeatletter
\makeatother


\begin{abstract} Let $(\fg,[p])$ be a finite dimensional restricted Lie algebra over a perfect field $\KK$ of characteristic $p\!\ge \!3$. By combining methods from recent work of Benson-Carlson \cite{BC20} with those of
\cite{CF21,Fa17} we obtain a description of the endotrivial $(\fg,[p])$-modules in case the underlying Lie algebra $\fg$ is supersolvable. For such $\fg$ and algebraically closed $\KK$, this yields a classification of the
indecomposable $(\fg,[p])$-modules of constant Jordan type with one non-projective block. \end{abstract}

\maketitle

\section*{Introduction} \label{S:Intro}
Let $\KK$ be a perfect field of characteristic $p\!\ge\!3$, $(\fg,[p])$ be a finite dimensional restricted Lie algebra with restricted enveloping algebra $U_0(\fg)$. In this article we employ geometric invariants that can be attached
to certain $U_0(\fg)$-modules in conjunction with the connectedness of projective varieties of $p$-subalgebras of $\fg$ to determine the so-called endotrivial $U_0(\fg)$-modules in case the Lie algebra $\fg$ is supersolvable.

Roughly speaking, a restricted Lie algebra $(\fg,[p])$ is an ordinary Lie algebra $\fg$ that is equipped with a $p$-th power map $\fg \lra \fg \ ; \ x \mapsto x^{[p]}$ with properties analogous to those of taking $p$-th powers in an
associative algebra. This is reflected in the definition of the restricted enveloping algebra $U_0(\fg)\!:=\!U(\fg)/(x^p\!-\!x^{[p]} ; x \in \fg)$ as a (finite dimensional) quotient of the ordinary enveloping algebra $U(\fg)$ of $\fg$.
Being a cocommutative Hopf algebra, $U_0(\fg)$ shares important properties with group algebras of finite groups. In particular, the category of $U_0(\fg)$-modules affords a duality and tensor products.

Our paper partly builds on recent work by Benson and Carlson \cite{BC20}, where somewhat different albeit related geometric methods were employed. In \cite{BC20} vector bundles were considered, while our method
ultimately hinges on morphisms to Grassmannians \cite{Fa17}. One key ingredient common to both approaches is the analysis of the \textit{nullcone}
\[ V(\fg) := \{x \in \fg \ ; \ x^{[p]}\!=\!0\},\]
which is a conical, closed subspace of the full affine space $\fg$.

In the more general context of finite group schemes techniques involving generalizations of the projectivized space $\PP(V(\fg))$ were introduced that in our situation amount to studying $U_0(\fg)$-modules $M$ by means of
their restrictions $(M|_{U_0(\KK x)})_{\KK x \in \PP(V(\fg))}$, cf.\ \cite{FP05,CFP08}. As each such $U_0(\KK x) \subseteq U_0(\fg)$ is canonically isomorphic to the truncated polynomial ring $\KK[X]/(X^p)$, there is a
decomposition
\[ M|_{U_0(\KK x)} \cong \bigoplus_{i=1}^p a_i(\KK x)_M[i]  \ \  ; \ \ a_i(\KK x)_M \in \NN_0,\]
with $[i]\!:=\!U_0(\KK x)/U_0(\KK x)x^i$ being the, up to isomorphism, unique indecomposable $U_0(\KK x)$-module of dimension $i$. If all the maps $\KK x \mapsto a_i(\KK x)_M$ are constant, then $M$ is said to have
{\it constant Jordan type} and we write $\Jt(M)\!:=\!\bigoplus_{i=1}^p a_i(M)[i]$. These modules have been studied partly because of their connections to vector bundles on $\PP(V(\fg))$. In another direction, one can ask to
what an extent modules $M$ of constant Jordan type are determined by their invariants $\Jt(M)$. In this framework, the study of modules $M$ with one non-projective block, that is, with $\Jt(M)\!=\![i]\!\oplus\!a_p(M)[p]$ for
some $i \in \{1,\ldots,p\!-\!1\}$, is a natural starting point. If $\KK$ is algebraically closed, it follows from \cite{Be10} and \cite{CF19} that we usually have $i \in \{1,p\!-\!1\}$, see Theorem \ref{NP2} for a precise statement.
Thanks to \cite{CFP08}, the resulting modules then are precisely the so-called endotrivial modules, that were initially investigated in the modular representation theory finite groups.

Regarding the history and relevance of endotrivial modules, the reader is referred to the introduction of \cite{BC20} and the references cited therein. Background material concerning restricted Lie algebras can be found in
\cite{SF}. All vector spaces are assumed to be finite dimensional.

We briefly outline the contents of our article. Given $(\fg,[p])$, we consider the Grassmannian $\Gr_2(\fg)$ of $2$-planes of the $\KK$-vector space $\fg$ and its closed subset
\[ \EE(2,\fg) := \{\fe \in \Gr_2(\fg) \ ; \ [\fe,\fe]\!=\!(0)\!=\!\fe^{[p]}\}\]
of two-dimensional elementary abelian $p$-subalgebras, cf.\ \cite{CFP15}. These spaces enable us to discern properties that cannot be detected at the level of $\PP(V(\fg))\!=\!\EE(1,\fg)$.

Our first main result reads as follows:

\bigskip

\begin{thm*} Suppose that $\KK$ is algebraically closed. Let $(\fu,[p])$ be a unipotent restricted Lie algebra. Then $\EE(2,\fu)$ is connected. \end{thm*}

\bigskip
\noindent
Here $(\fu,[p])$ is called \textit{unipotent} or \textit{$[p]$-nilpotent}, provided there exists $n \in \NN_0$ such that $x^{[p]^n}\!=\!0$ for all $x \in \fu$. The former terminology derives from the fact that $\fu$ corresponds
to a unipotent infinitesimal group scheme of height $1$. By Engel's Theorem unipotent restricted Lie algebras are nilpotent, and the example \cite[p.117f]{CF21} shows that Theorem A may fail in the latter context.
Moreover, the three dimensional unipotent Heisenberg algebra shows that the result may also fail in case $\Char(\KK)\!=\!2$. By contrast, early work of Carlson \cite{Ca84} and Friedlander-Parshall \cite{FP86} ensures
that the space $\EE(1,\fg)$ is connected for an arbitrary restricted Lie algebra $(\fg,[p])$.

For elementary abelian Lie algebras $\fe$, $U_0(\fe)$-modules $M$ of constant Jordan type induce morphisms on $\PP(\fe)$ that take values in certain Grassmannians $\Gr_{\rk(M)}(M)$. The common degree of the
describing homogeneous polynomials is called the degree $\deg(M)$ of $M$ (see \cite[\S 4.1]{Fa17}). Given a $U_0(\fg)$-module $M$ of constant Jordan type, the degree function
\[ \msdeg_M : \EE(2,\fg) \lra \NN_0 \ \ ; \ \ \fe \mapsto \deg(M|_{U_0(\fe)})\]
is locally constant (cf.\ \cite{CF21}) and thus constant whenever $\EE(2,\fg)$ is connected.

Let $\modd U_0(\fg)$ be the category of finite dimensional $U_0(\fg)$-modules and denote by $X(\fg)$ the group of \textit{characters} of $(\fg,[p])$, that is, the set of algebra homomorphisms $U_0(\fg) \lra \KK$. (This is just
the group of group-like elements of the dual Hopf algebra $U_0(\fg)^\ast$.) The elements of $X(\fg)$ are determined by their restrictions to  $\fg$, so that we may think of them as linear forms $\lambda : \fg \lra \KK$ such that
$\lambda([\fg,\fg])\!=\!(0)$ and $\lambda(x^{[p]})\!=\!\lambda(x)^p$ for all $x \in \fg$. Given $\lambda \in X(\fg)$, we let $\KK_\lambda$ be the one-dimensional $U_0(\fg)$-module defined by $\lambda$. By relating degree
functions to syzygy functions of endotrivial modules and combining Theorem A  with results from \cite{BC20}, we show that endotrivial modules over supersolvable restricted Lie algebras are stably isomorphic to Heller shifts
of one-dimensional modules:

\bigskip

\begin{thm**}  Let $(\fg,[p])$ be a supersolvable restricted Lie algebra. If $M \in \modd U_0(\fg)$ is endotrivial, then there exist $\lambda \in X(\fg)$ and $n \in \ZZ$ such that $M\cong
\Omega^n_{U_0(\fg)}(\KK_\lambda)\!\oplus\!({\rm proj.})$. \end{thm**}

\bigskip
\noindent
In case $(\fg,[p])$ is unipotent and $p\!\ge\!5$, Theorem B was proved in \cite{BC20} for arbitrary ground fields.\footnote{In \cite{BC20} unipotent restricted Lie algebras are being referred to as ``nilpotent''.}  In this
context, our proof of Theorem B also does not necessitate $\KK$ to be perfect. As observed in \cite{BC20}, Theorem B may already fail for unipotent Lie algebras in case $\Char(\KK)\!=\!2$. In case $\fg\!:=\!\Lie(G)$
is the Lie algebra of an algebraic group $G$, our result determines the structure of the restrictions of endotrivial $U_0(\fg)$-modules to Borel subalgebras of $\fg$.

In the final section, we show that modules $M$ of constant Jordan type $\Jt(M)\!=\![i]\!\oplus\!a_p(M)[p]$ with $i \in \{2,\ldots,p\!-\!2\}$ only occur for certain algebras $(\fg,[p])$, whose representation theory is well understood.
For supersolvable $\fg$, this yields the classification of the $U_0(\fg)$-modules of constant Jordan type with one non-projective block.

\bigskip

\section{Connectedness of $\EE(2,\fu)$} \label{S:Connect}
Throughout this section, $\KK$ is assumed to be algebraically closed and $(\fu,[p])$ denotes a unipotent restricted Lie algebra. If $\fu\!\ne\!(0)$, then the nullcone $V(C(\fu))$ of the center $C(\fu)$ of $\fu$ is not zero.
Given $z_0 \in V(C(\fu))\!\smallsetminus\!\{0\}$,
\[ \EE(2,\fu)_{z_0} := \{\fe \in \EE(2,\fu) \ ; \ z_0 \in \fe\}.\]
is a closed subset of $\EE(2,\fu)$.

\bigskip

\begin{Lemma} \label{Connect1} Suppose that $V(\fu)$ is irreducible and of dimension $\dim V(\fu)\!\ge\!2$. Then $\EE(2,\fu)$ is non-empty and connected. \end{Lemma}

\begin{proof} Let $z_0 \in V(C(\fu))\!\smallsetminus\!\{0\}$. By assumption, $X\!:=\!V(\fu)\!\smallsetminus\!\KK z_0$ is a non-empty open and hence irreducible conical subvariety of $V(\fu)$. Consequently,
$\PP(X) \subseteq \PP(V(\fu))$ is a quasi-projective, irreducible variety.

The linear map $\fu \mapsto \bigwedge^2(\fu) \ ; \ x \mapsto x\wedge z_0$ defines a homogeneous morphism
\[ V(\fu) \lra \bigwedge^2(\fu)\]
of conical affine varieties. It thus follows from \cite[(1.3.1)]{Fa17} that the composite of
\[ \varphi :  \PP(X) \lra \EE(2,\fu)_{z_0} \ \ ; \ \ \KK x \mapsto \KK x\!\oplus\!\KK z_0\]
with the Pl\"ucker embedding $\Gr_2(\fu) \lra \PP(\bigwedge^2(\fu))$ is a morphism. Hence $\varphi$ is a surjective morphism, so that the projective variety $\EE(2,\fu)_{z_0}$ is irreducible.

Let $\cX_0$ be the connected component of $\EE(2,\fu)$ that contains $\EE(2,\fu)_{z_0}$. Given $\fe \in \EE(2,\fu)\!\smallsetminus\!\EE(2,\fu)_{z_0}$, we consider the
abelian three-dimensional $p$-subalgebra $\fa\!:=\!\fe\!\oplus\!\KK z_0$. Since $\EE(2,\fa)\!=\!\Gr_2(\fa)$ is irreducible and $\EE(2,\fa)\cap \EE(2,\fu)_{z_0}\!\ne\!\emptyset$, it follows that
$\EE(2,\fa) \subseteq \cX_0$, whence $\fe \in \cX_0$. As a result, $\EE(2,\fu)\!=\!\cX_0$ is connected.  \end{proof}

\bigskip
\noindent
A $p$-subalgebra $\fm \subseteq \fu$ is called {\it maximal}, if it is a maximal element (relative to the partial order given by inclusion $\subseteq$) of the set of all proper $p$-subalgebras of $\fu$. We let
\[\Max_p(\fu)\!:=\!\{\fm \subseteq \fu \ ; \ \fm \ \text{maximal $p$-subalgebra such that} \ \EE(2,\fm)\!\ne\!\emptyset\}.\]
Suppose that $\dim_\KK\fu\!\ge\!3$. Then every $\fe \in \EE(2,\fu)$ is contained in a maximal $p$-subalgebra, so that
\[ \EE(2,\fu) = \bigcup_{\fm \in \Max_p(\fu)} \EE(2,\fm).\]
Let $\fm \subsetneq \fu$ be a maximal $p$-subalgebra of $\fu$. Since $\fu$ is unipotent, Engel's Theorem implies that the normalizer $\Nor_{\fu}(\fm)$ of $\fm$ in $\fu$ is a $p$-subalgebra of $\fu$ that properly contains
$\fm$. By choice of $\fm$, we have $\Nor_{\fu}(\fm)\!=\!\fu$, so that $\fm$ is a $p$-ideal of  $\fu$. Thus, $\fu/\fm$ is a unipotent restricted Lie algebra that does not contain any proper $p$-subalgebras. Consequently,
$\dim_\KK\fu/\fm\!=\!1$. In particular, every $\fm \in \Max_p(\fu)$ is a $p$-ideal of codimension $1$.

Given $x \in \fu$, we let $(\KK x)_p\!:=\!\sum_{j \ge 0} \KK x^{[p]^j}$ be the smallest $p$-subalgebra of $\fu$ containing $x$. Such $p$-subalgebras, which are necessarily abelian, are referred to as \textit{cyclic}.

\bigskip

\begin{Lemma} \label{Connect2} Let $\fm,\fn \in \Max_p(\fu)$ be such that $\EE(2,\fm)\cap \EE(2,\fn)\!=\!\emptyset$. Then the following statements hold:
\begin{enumerate}
\item There are $x \in V(\fm)$ and $y \in V(\fn)$ such that $\KK x\!\oplus\!\fn\!=\!\fu\!=\! \KK y\!\oplus\!\fm$.
\item The $p$-ideal $\fm\cap \fn$ is cyclic.
\item $\dim_\KK( \fm\cap \fn)/(C(\fu)\cap\fm\cap\fn)\!\le\!1$.
\item $\dim_\KK \fm/(C(\fu)\cap \fm) \!\le \!2$.
\item $\dim_\KK \fu/C(\fu)\!\le\!3$.
\end{enumerate}\end{Lemma}

\begin{proof} (1) Suppose that $V(\fm) \subseteq \fn$. If $\fe \in \EE(2,\fm)$, then $\fe \subseteq V(\fm) \subseteq \fn$, so that $\fe \in \EE(2,\fn)$, a contradiction. Since the $p$-ideals $\fm\!\ne\!\fn$ have
codimension $1$, the assertion follows.

(2),(3) Since $\fm\cap\fn$ is unipotent with $\EE(2,\fm\cap\fn)\!=\!\EE(2,\fm)\cap\EE(2,\fn)\!=\!\emptyset$, we have $\dim V(\fm\cap\fn)\!\le\!1$ and the $p$-ideal $\fm\cap\fn$ is cyclic (see for instance \cite[(4.3)]{Fa95}).
Consequently, $(\fm\cap\fn)^{[p]}\subseteq C(\fu)$, so that $C(\fu)\cap \fm\cap \fn$ has codimension $\le 1$ in $\fm\cap \fn$.

(4),(5) We have
\[ \dim_\KK\fm\!-\!\dim_\KK C(\fu)\cap \fm \le  \dim_\KK \fm\!-\!\dim_\KK C(\fu)\cap \fm\cap \fn = 1\!+\!\dim_\KK\fm\cap\fn\!-\!\dim_\KK C(\fu)\cap \fm\cap \fn \le 2.\]
Similarly,
\[ \dim_\KK\fu\!-\!\dim_\KK C(\fu) \le  \dim_\KK\fu\!-\!\dim_\KK C(\fu)\cap \fm = 1\!+\!\dim_\KK\fm\!-\!\dim_\KK C(\fu)\cap \fm \le 3,\]
as asserted. \end{proof}

\bigskip
\noindent
We let $(\fu^n)_{n \in \NN}$ be the lower central series of $\fu$, whose members are defined inductively via $\fu^1\!:=\!\fu$ and $\fu^{n+1}\!:=\![\fu,\fu^n]$ for all $n \in \NN$.

\bigskip

\begin{Corollary} \label{Connect3} Let $\fm,\fn \in \Max_p(\fu)$ be such that $\EE(2,\fm)\cap \EE(2,\fn)\!=\!\emptyset$. Then the following statements hold:
\begin{enumerate}
\item If $\dim_\KK \fu/C(\fu)\!\le\!2$, then $\EE(2,\fu)$ is connected.
\item If $p\!\ge\!5$, then $\EE(2,\fu)$ is connected. \end{enumerate} \end{Corollary}

\begin{proof} Suppose that the $p$-map is semilinear, that is, $(\alpha x\!+\!y)^{[p]}\!=\!\alpha^px^{[p]}\!+\!y^{[p]}$ for all $x,y \in \fu$ and $\alpha \in \KK$. Then $V(\fu)$ is a subspace of dimension $\ge\!2$ and Lemma \ref{Connect1} shows that $\EE(2,\fu)$ is connected.

(1) By assumption, $\fu^2 \subseteq C(\fu)$, whence $\fu^3\!=\!(0)$. In view of \cite[(2.1.2)]{SF}, the map $[p]$ is semilinear.

(2)  Thanks to Lemma \ref{Connect2}(1), we have $\fu\!=\!ky\!\oplus \!\fm$ for some $y\in V(\fn)$. Given
\[ a = b \oplus c \in \KK y\!\oplus \fm, \ \ \ \ \ (b \in \KK y \subseteq \fn, c \in \fm),\]
Jacobson's formula \cite[\S 2.1]{SF} implies
\[ a^{[p]} = c^{[p]}\!+\!\sum_{i=1}^{p-1} s_i(b,c),\]
where the right-hand side belongs to $p$-th constituent $\fs^p$ of the lower central series of the subalgebra $\fs$ generated by $b$ and $c$. By choice of $b$ and $c$, we have $\fs^2\!=\!\sum_{i\ge2}(\KK b\oplus\!\KK c)^i
\subseteq \fm\cap \fn$, while Lemma \ref{Connect2}(3) implies $\fs^3 \subseteq C(\fu)$. Consequently, $\fs^4\!=\!(0)$. As $p\!\ge\!5$, \cite[(2.1.2)]{SF} yields
\[ a^{[p]} = c^{[p]}.\]
Lemma \ref{Connect2}(4) ensures that $\fm/(C(\fu)\cap \fm)$ is abelian, so that we have $\fm^2 \subseteq C(\fu)$ and $\fm^3\!=\!(0)$. Consequently, $[p]|_\fm$ is semilinear. In sum, we have shown that the $p$-map is
semilinear, so that $\EE(2,\fu)$ is connected. \end{proof}

\bigskip

\begin{Lemma} \label{Connect4} Suppose that $\fu$ has minimal dimension subject to $\EE(2,\fu)$ being disconnected. Then the following statements hold:
\begin{enumerate}
\item There exist $\fm,\fn \in \Max_p(\fu)$ such that $\EE(2,\fm)\cap\EE(2,\fn)\!=\!\emptyset$.
\item We have $C(\fu)\!=\!(\fm\cap\fn)^{[p]}$. In particular, $C(\fu)$ is cyclic. \end{enumerate} \end{Lemma}

\begin{proof} (1) By assumption, we have $\dim_\KK \fu\!\ge\!3$, so that $\EE(2,\fu)\!=\!\bigcup_{\fm \in \Max_p(\fu)}\EE(2,\fm)$.

Let $\fm \in \Max_p(\fu)$. By choice of $\fu$, the space $\EE(2,\fm)$ is connected. If $\EE(2,\fm)\cap\EE(2,\fn)\!\ne\!\emptyset$ for all $\fm,\fn \in \Max_p(\fu)$, then $\EE(2,\fu)$ is connected, a contradiction.

(2) In view of Lemma \ref{Connect2}(5) and Corollary \ref{Connect3}(1), we conclude that $\dim_\KK \fu/C(\fu)\!=\!3$. Moreover, Lemma \ref{Connect2}(3) yields $\dim_\KK\fu/(C(\fu)\cap \fm\cap\fn)\le 3$, whence
$C(\fu) \subseteq \fm\cap\fn$. Since $\dim_\KK\fu/(\fm\cap\fn)\!=\!2$, this inclusion is strict. On the other hand, $\fm\cap\fn$ being abelian forces $(\fm\cap\fn)^{[p]}\subseteq C(\fu)$, so that Lemma \ref{Connect2}(2)
yields $C(\fu)\!=\!(\fm\cap\fn)^{[p]}$. \end{proof}

\bigskip
\noindent
\textbf{Proof of Theorem A}:  We let $\fu$ be of minimal dimension subject to $\EE(2,\fu)$ being disconnected.

Lemma \ref{Connect4}(1) provides $\fm,\fn \in \Max_p(\fu)$ such that $\EE(2,\fm)\cap\EE(2,\fn)\!=\!\emptyset$. If $p\!\ge\!5$, then Corollary \ref{Connect3}(2) yields a contradiction, so that $p\!=\!3$.

Lemma \ref{Connect2}(1) provides $x \in V(\fm)$ and $y \in V(\fn)$ such that $\KK x\!\oplus\!\fn\!=\!\fu\!=\!\KK y\!\oplus\!\fm$. It follows that
\[ \fu = \KK x\!\oplus\!\KK y\!\oplus\!(\fm\cap\fn).\]
Then $z\!:=\![x,y] \in \fm\cap \fn$. If $z \in C(\fu)$, then Lemma \ref{Connect2}(3) yields $\fu^2\!=\!\KK z\!+\![x,\fm\cap\fn]\!+\![y,\fm\cap\fn]\!+\![\fm\cap\fn,\fm\cap\fn] \subseteq C(\fu)$, whence $\fu^3\!=\!(0)$.
Hence the map $[3]$ is semilinear and $\EE(2,\fu)$ is connected, a contradiction. Lemma \ref{Connect4}(2) yields $(\fm\cap \fn)^{[3]}\!=\!C(\fu)$, so that Lemma \ref{Connect2}(2) implies $\fm\cap \fn\!=\!(\KK z)_3$.

The remainder of the proof rests on an analysis of Jacobson's formula
\[ (a\!+\!b)^{[3]} = a^{[3]}\!+\!b^{[3]}\!+\!\sum_{i=1}^2 s_i(a,b),\]
see \cite[\S 2.1]{SF}. Let $T$ be an indeterminate over $\KK$ and consider the restricted Lie algebra $\fu\!\otimes_\KK\!\KK[T]$. Recall that
\[ \ad (\alpha x\otimes T\!+\!\beta y\otimes1)^2(\alpha x\otimes 1) = \sum_{i=1}^2 is_i(\alpha x,\beta y)\otimes T^{i-1},\]
while
\[ \ad (\alpha x\otimes T\!+\!\beta y\otimes1)^2(\alpha x\otimes 1) = -\ad (\alpha x\otimes T\!+\!\beta y\otimes1)(\alpha\beta z\otimes 1) = -\alpha^2\beta [x,z]\otimes T\!-\!\alpha\beta^2[y,z]\!\otimes 1,\]
so that $s_1(\alpha x,\beta y)\!=\!-\alpha\beta^2[y,z]$ and $s_2(\alpha x,\beta y)\!=\!\alpha^2\beta[x,z]$. This yields
\[ (\alpha x\!+\!\beta y)^{[3]} = \alpha^2\beta[x,z]\!-\!\alpha\beta^2[y,z] \ \ \ \ \ \forall \ \alpha,\beta \in \KK.\]
Given $u \in \fm\cap\fn$, Lemma \ref{Connect2}(3) implies $s_i(\alpha x\!+\!\beta y,u)\!=\!0$ for $1\!\le\!i\!\le\!2$, so that
\[ (\ast) \ \ \ \ \ \ \ \ (\alpha x\!+\!\beta y\!+\!u)^{[3]} = \alpha^2\beta[x,z]\!-\!\alpha\beta^2[y,z]\!+\!u^{[3]} \ \ \ \ \ \forall \ \alpha,\beta \in \KK, u \in \fm\cap\fn.\]
Let $\fm\cap\fn\!=\!\bigoplus_{i=0}^{n+1}\KK z^{[3]^i}$, where $z^{[3]^{n+1}}\!\ne\!0\!=\!z^{[3]^{n+2}}$. In view of Lemma \ref{Connect4}(2), the map
\[\bigoplus_{i=0}^{n}\KK z^{[3]^i} \lra C(\fu) \  \   ;  \  \  a  \mapsto a^{[3]}\]
is bijective. By Lemma \ref{Connect2}(3), we have $[x,z],[y,z] \in C(\fu)$, and we can thus find unique elements $a,b \in \bigoplus_{i=0}^{n}\KK z^{[3]^i}$ such that
\[ a^{[3]} = [x,z] \  \  \text{and} \ \ b^{[3]} = [y,z].\]
Let $u \in V(\fu)$. Then there are $\alpha,\beta,\epsilon \in \KK$ and $s \in \bigoplus_{i=0}^{n}\KK z^{[3]^i}$ such that
\[ u = \alpha x\!+\!\beta y\!+\!s\!+\!\epsilon z_0,\]
where $z_0\!:=\!z^{[3]^{n+1}}$ and $V(C(\fu))\!=\!\KK z_0$. Since $u \in V(\fu)$, identity ($\ast$) yields
\[ 0 = (\alpha x\!+\!\beta y\!+\!s\!+\!\epsilon z_0)^{[3]} = \alpha^2\beta[x,z]\!-\!\alpha\beta^2[y,z]\!+\!s^{[3]}.\]
As $\alpha^2\beta[x,z]\!-\!\alpha\beta^2[y,z]\in C(\fu)$, this implies the existence of $\gamma,\delta \in \KK$ such that
\[ s=\!\gamma a\!+\!\delta b.\]
Hence
\[ 0 = (\alpha^2\beta\!+\gamma^3)[x,z]\!+\!(\delta^3\!-\!\alpha\beta^2)[y,z].\]
It follows that
\[ (\ast\ast) \ \ \ \ \ \ \ V(\fu) = \{ \alpha x\!+\!\beta y\!+\!\gamma a\!+\!\delta b\!+\!\epsilon z_0 \ ; \ (\alpha^2\beta\!+\gamma^3)[x,z]\!+\!(\delta^3\!-\!\alpha\beta^2)[y,z]\!=\!0\}.\]
Two cases arise:

\medskip
(a) {\it $[x,z]$ and $[y,z]$ are linearly independent}.

\smallskip
\noindent
Let
\[X := \{\!\alpha x\!+\!\beta y\!+\!\gamma a\!+\!\delta b\in \fu \ ; \ \alpha^2\beta\!+\!\gamma^3\!=\! 0, \delta^3\!-\!\alpha\beta^2\!=\!0\}\]
and
\[Y := \{\alpha x\!+\!\beta y\!+\!\gamma a \in \fu \ ; \ \alpha^2\beta\!+\!\gamma^3\!=\! 0\}.\]
Then $X,Y$ are conical closed subsets of $\fu$. In view of ($\ast\ast$), we have an isomorphism $X\!\times\!\KK z_0 \stackrel{\sim}{\lra} V(\fu)$. The canonical projection induces a surjective morphism
\[ \pr : \PP(X) \lra \PP(Y) \ \ ; \ \  \KK(\alpha x\!+\!\beta y\!+\!\gamma a\!+\!\delta b) \mapsto \KK(\alpha x\!+\!\beta y\!+\!\gamma a)\]
such that $\pr^{-1}(\fv)$ is a singleton for all $\fv \in \PP(Y)$. Since $Y$ is irreducible, so is $\PP(Y)$, and it follows from \cite[Chap.I,\S6,Thm.8]{Sh1} that $\PP(X)$ is irreducible (of dimension $\dim \PP(X)\!=\!1$).
Since each irreducible component of $X$ is conical, it follows that $X$ is irreducible. Thus, $V(\fu)$ is irreducible as well.

\medskip
(b) {\it There is $\omega_0 \in \KK$ such that $[y,z]\!=\!\omega_0[x,z]$}.\footnote{This includes the example of \cite[\S 7]{BC20} as a special case.}

\smallskip
\noindent
Then we have $b \in \KK a$ and, for $u \in V(\fu)$, we write
\[ u = \alpha x\!+\!\beta y\!+\!\gamma a\!+\!\epsilon z_0.\]
We obtain
\[ 0 = u^{[3]} = (\alpha^2\beta\!+\!\gamma^3)[x,z]\!-\!\alpha\beta^2 \omega_0[x,z] = (\alpha^2\beta\!-\!\alpha\beta^2\omega_0\!+\gamma^3)[x,z].\]
Let
\[X := \{\alpha x\!+\!\beta y\!+\!\gamma a \in \fu \ ; \ \alpha^2\beta\!-\alpha\beta^2\omega_0+\!\gamma^3\!=\!0\}\]
Then $X$ is a conical, closed, irreducible subset of $V(\fu)$, and, if $[x,z]\!\ne\!0$, there is a canonical isomorphism $X\!\times\!\KK z_0 \stackrel{\sim}{\lra} V(\fu)$. Hence $V(\fu)$ is irreducible. This also holds
for $[x,z]\!=\!0$.

\medskip
\noindent
In either case, Lemma \ref{Connect1} provides a contradiction, thereby proving our theorem. \hfill $\square$

\bigskip

\section{Syzygy functions for endotrivial modules} \label{S:Syz}
Let $(\fg,[p])$ a restricted Lie algebra with restricted universal enveloping algebra $U_0(\fg)$. Given $M \in \modd U_0(\fg)$, and $n \in \ZZ$, we consider the $n$-th Heller shift $\Omega^n_{U_0(\fg)}(M)$, which is
computable from a minimal projective resolution (for $n\!\ge\!0$) or a minimal injective resolution of $M$ (for $n\!\le\!0$). We refer the reader to \cite[\S IV.3]{ARS95} for basic properties of $\Omega^n_{U_0(\fg)}$.

If $\fe$ is two-dimensional elementary abelian Lie algebra, i.e., $[\fe,\fe]\!=\!(0)$ and $[p]\!=\!0$, then \cite[(6.9)]{CFP08} yields
\[ (\star) \ \ \ \ \ \dim_\KK\Omega_{U_0(\fe)}^{2n}(\KK)=np^2\!+\!1 \ \ ;  \ \ \dim_\KK\Omega_{U_0(\fe)}^{2n-1}(\KK)=np^2\!-\!1.\]
Recall that $M \in \modd U_0(\fg)$ is referred to as {\it endotrivial}, provided
\[ M\!\otimes_\KK\!M^\ast \cong \KK \!\oplus\!({\rm proj.}).\]
In other words, the $U_0(\fg)$-module $M\!\otimes_\KK\!M^\ast$ is stably equivalent to the trivial module module $\KK\!=\!\KK_\varepsilon$, with $\varepsilon \in X(\fg)$ being the co-unit of the Hopf algebra $U_0(\fg)$.
Let $M$ be an endotrivial $U_0(\fg)$-module. If $\fe\in\EE(2,\fg)$ is a two-dimensional elementary abelian $p$-subalgebra, then the restriction of $M|_{U_0(\fe)}$ of $M$ to $U_0(\fe)$ is also endotrivial. Thanks to
\cite[(4.1)]{BC20} and ($\star$), there exists a unique integer $s_M(\fe)\in\ZZ$ such that
\[M|_{U_0(\fe)} \cong\Omega_{U_0(\fe)}^{s_M(\fe)}(\KK)\oplus (\rm proj.).\]

\begin{Definition} Let $(\fg,[p])$ be a restricted Lie algebra such that $\EE(2,\fg)\neq\emptyset$, $M$ be an endotrivial $U_0(\fg)$-module. Then
\[s_M : \EE(2,\fg)\lra \ZZ \]
is called the \textit{syzygy function} of $M$. \end{Definition}

\bigskip
\noindent
Using this notion we record the following consequence of \cite[(4.1),(4.2)]{BC20}:

\bigskip

\begin{Proposition} \label{Syz1} Let $\KK$ be algebraically closed. Suppose that $(\fu,[p])$ is unipotent, and let $z_0 \in V(C(\fu))\!\smallsetminus\!\{0\}$. Let $M$ be an endotrivial $U_0(\fu)$-module. If $\EE(2,\fu)_{z_0}\!=\!
\emptyset$\footnote{In this case, $\dim V(\fu)\!\le\!1$, so that $\fu$ is cyclic,} or $s_M : \EE(2,\fu)_{z_0} \lra \ZZ$ is constant, then there is $n \in \NN$ such that $M\cong \Omega^n_{U_0(\fu)}(\KK)\!\oplus\!({\rm proj.}).$
\end{Proposition}

\bigskip
\noindent
The remainder of this section addresses the question, when syzygy functions are constant. Suppose that $\KK$ is algebraically closed. We let $\CJT(\fg)$ be the full subcategory of $U_0(\fg)$-modules of constant Jordan
type.  Endotrivial modules are known to have constant Jordan type, cf.\ \cite{CFP08}. For $\fe \in \EE(2,\fg)$ and $N \in \CJT(\fe)$, we define the \textit{generic kernel} of $N$ via
\[ \fK(N) := \sum_{x \in \PP(\fe)} \ker x_N\]
and refer the reader to \cite[\S7]{CFS11} for more details (in particular \cite[(7.7)]{CFS11}).

\bigskip
\noindent
Given $M \in \CJT(\fg)$, there exists $\rk(M) \in \NN_0$ such that
\[ \rk(x_M) = \rk(M) \ \ \ \ \ \forall \ x \in V(\fg)\!\smallsetminus\!\{0\}.\]
Here $x_M : M \lra M$ is the left multiplication effected by $x \in V(\fg) \subseteq U_0(\fg)$.

In the following result, we recall those properties of the degree function $\msdeg_M$ that are relevant for our purposes: \footnote{The degree function $\msdeg_M$ is just the $1$-degree function of \cite{Fa17}.}

\begin{Proposition} \label{Syz2} Let $\KK$ be algebraically closed, $M \in \CJT(\fg)$, $\fe \in \EE(2,\fg)$.
\begin{enumerate}
\item We have $\msdeg_M(\fe)\!=\!\dim_\KK M/\fK(M|_{U_0(\fe)})$.
\item We have $M^\ast \in \CJT(\fg)$ and $\msdeg_M(\fe)\!+\!\msdeg_{M^\ast}(\fe)\!=\!\rk(M)$. \end{enumerate} \end{Proposition}

\begin{proof} (1) This is a direct consequence of \cite[(6.2.8)]{Fa17}, observing \cite[(7.10)]{CFS11}.

(2) This follows from \cite[Thm.A]{Fa17}. \end{proof}

\bigskip

\begin{Proposition}\label{Syz3} Suppose that $\KK$ is algebraically closed and that $\EE(2,\fg)\!\ne\!\emptyset$. If $M \in \modd U_0(\fg)$ is endotrivial, then we have
\begin{equation*}
2p\msdeg_M(\fe)= \left\{
\begin{array}{ll}
(p\!-\!1)(\dim_\KK M\!-\!1\!-\!s_M(\fe)p) & \text{if}  \ \dim_\KK M\!\equiv\! 1~{\rm mod}(p),\\
(p\!-\!1)(\dim_\KK M\!+\!1)\!-\!(s_M(\fe)\!+\!1)p&\text{if} \ \dim_\KK M\!\equiv\!-\!1~{\rm mod}(p),\\\end{array}
\right. \end{equation*}
for all $\fe\in\EE(2,\fg)$. \end{Proposition}

\begin{proof} We only consider the case where the dimension of $M$ is congruent to one modulo $p$. The other case can be treated similarly. We write $\dim_\KK M\!=\!mp\!+\!1$ for some
integer $m \in\NN_0$.

Given $\fe \in \EE(2,\fg)$, the algebra $U_0(\fe)\cong \KK[X,Y]/(X^p,Y^p)$ is local, and there thus exists an integer $n_M(\fe)\in\NN_0$ such that
\[ M|_{U_0(\fe)}\cong\Omega_{U_0(\fe)}^{s_M(\fe)}(\KK)\oplus U_0(\fe)^{n_M(\fe)}.\]
We first assume that $s_M(\fe)\!\geq\!0$ and put $M_1\!:=\!\Omega_{U_0(\fe)}^{s_M(\fe)}(\KK)$ and $M_2\!:=\!U_0(\fe)^{n_M(\fe)}$. In view of ($\star$), $s_M(\fe)$ is even and direct computation shows that
\[ n_M(\fe) = \frac{\dim_\KK M\!-\!\dim_\KK M_1}{p^2} = \frac{2m\!-\!s_M(\fe)p}{2p}.\]
Since $U_0(\fe)^{n_M(e)}$ is self-dual, Proposition \ref{Syz2}(2), applied to $\deg_{M_2} : \EE(2,\fe) \lra \NN_0$, gives
\[ (\dagger)  \ \ \ \ \ \msdeg_{M_2}(\fe)= n_M(\fe)\frac{p(p\!-\!1)}{2} = \frac{(2m\!-\!s_M(\fe)p)(p\!-\!1)}{4}.\]
Observing \cite[(7.12)]{CFS11}, we arrive at
\[ (\dagger\dagger) \ \ \ \ \ \msdeg_{M_1}(\fe) = \dim_\KK M_1/\fK(M_1) = \frac{s_M(\fe)(p\!-\!1)(p\!-\!2)}{4}.\]
Now ($\dagger$) and ($\dagger\dagger$) yield the desired formula.

If $s_M(\fe)\!<\!0$, we consider the endotrivial $U_0(\fg)$-module $M^\ast$. Since $(\Omega_{U_0(\fe)}^{s_M(\fe)}(\KK))^\ast\cong\Omega_{U_0(\fe)}^{-s_M(\fe)}(\KK)$, we obtain $s_{M^\ast}(\fe)\!=\!-s_M(\fe)$ and
\cite[(6.9)]{CFP08} yields $\rk(M)\!=\!\frac{p\!-\!1}{p}(\dim_\KK M\!-\!1)$. The result now follows from Proposition \ref{Syz2}(2). \end{proof}

\bigskip

\begin{Corollary}\label{Syz4} Suppose that $\KK$ is algebraically closed and that $\EE(2,\fg)\!\neq\!\emptyset$. Let $M$ be an endotrivial $U_0(\fg)$-module. If $\EE(2,\fg)$ is connected, then $s_M$ is constant.
\end{Corollary}

\begin{proof} This is a direct consequence of Proposition \ref{Syz3} and \cite[(4.1.1)(3)]{CF21}. \end{proof}

\bigskip

\section{Endotrivial modules for supersolvable restricted Lie algebras} \label{S:Endo}
A Lie algebra $\fg$ is called {\it supersolvable}, provided its derived algebra $[\fg,\fg]$ is nilpotent. We say that $(\fg,[p])$ is a \textit{torus}, if
\begin{enumerate}
\item[(a)] $\fg$ is abelian, and
\item[(b)] every $x \in \fg$ is semisimple, that is, $x \in (\KK x^{[p]})_p$. \end{enumerate}
Restricted enveloping algebras of tori are semisimple \cite[Thm.]{Ho54}, so that every nonzero $U_0(\fg)$-module is endotrivial in that case.

We denote by $T(\fg)$ the unique maximal torus of $C(\fg)$. For convenience, we record the following result:

\bigskip

\begin{Lemma} \label{Endo1} Suppose that $\KK$ is algebraically closed and let $\pi : \fg \lra \fg/T(\fg)$ be the canonical projection. The following statements hold:
\begin{enumerate}
\item $\pi$ induces a bijective morphism $\pi : V(\fg) \lra V(\fg/T(\fg))$ of affine varieties.
\item The restriction of the induced homomorphism $\pi : U_0(\fg) \lra U_0(\fg/T(\fg))$ provides an isomorphism $\cB_0(\fg) \lra \cB_0(\fg/T(\fg))$ between the principal blocks of $U_0(\fg)$ and $U_0(\fg/T(\fg))$.
\end{enumerate} \end{Lemma}

\begin{proof} (1) We only verify the surjectivity. If $a\!=\!\pi(x) \in V(\fg/T(\fg))$, then $x^{[p]}\in T(\fg)$. As $T(\fg)$ is a torus, we can find $t \in T(\fg)$ such that $x^{[p]}\!=\!t^{[p]}$. Consequently, $x\!-\!t \in V(\fg)$, while
$a\!=\!\pi(x\!-\!t)$.

(2) This follows from \cite[(1.1)]{Fa06}.  \end{proof}

\bigskip
\noindent
\textbf{Proof of Theorem B.} We first prove the following claim:

\medskip
(i) \textit{Suppose that $\KK$ is algebraically closed. Let $M$ be an indecomposable endotrivial $U_0(\fg)$-module. Then there are $\lambda \in X(\fg)$ and $n \in \ZZ$ such that $M \cong
\Omega^n_{U_0(\fg)}(\KK_\lambda)$.}

\smallskip
\noindent
We proceed in several steps, noting that (i) trivially holds in case $(\fg,[p])$ is a torus.

(1) \textit{Suppose that $(\fg,[p])$ is unipotent}. Theorem A implies that $\EE(2,\fg)$ is connected. If $\dim V(\fg)\!\ge\!2$, then $\EE(2,\fg)_{z_0}\!\ne\!\emptyset$ for $z_0\in V(C(\fg))\!\smallsetminus\!\{0\}$ and Corollary
\ref{Syz4} shows that the syzygy function $s_M : \EE(2,\fg) \lra \ZZ$ is constant. Thanks to Proposition \ref{Syz1}, we have $M \cong \Omega^n_{U_0(\fg)}(\KK)$ for some $n \in \ZZ$. Alternatively, $\EE(2,\fg)\!=
\!\emptyset$ and either $\fg\!=\!(0)$ is a torus, or Proposition \ref{Syz1} applies.

(2) \textit{Suppose that $(\fg,[p])$ is trigonalizable, that is, every simple $U_0(\fg)$-module is one-dimensional}. By general theory, we have
\[ \fg \cong \ft\ltimes \fu,\]
where $\ft$ is a torus and $\fu$ is unipotent. Since $M|_{U_0(\fu)}$ is endotrivial, (1) provides an isomorphism
\[ M|_{U_0(\fu)} \cong \Omega^n_{U_0(\fu)}(\KK)\oplus U_0(\fu)^\ell\]
for some $n \in \ZZ$ and $\ell \in \NN_0$. Considering $N\!:=\!\Omega^{-n}_{U_0(\fg)}(M)$, we thus arrive at
\[ N|_{U_0(\fu)} \cong \KK\!\oplus U_0(\fu)^{\ell'} \ \ \ \ \ (\ell' \in \NN_0).\]
As $N$ is indecomposable, \cite[(2.1.2)]{Fa09} ensures that $N|_{U_0(\fu)}$ is indecomposable. Consequently, $N|_{U_0(\fu)} \cong \KK$ is one-dimensional, so that $N \cong \KK_\lambda$ for some $\lambda \in X(\fg)$.
As a result, $M \cong \Omega^n_{U_0(\fg)}(N) \cong  \Omega^n_{U_0(\fg)}(\KK_\lambda)$.

(3) \textit{Let $(\fg,[p])$ be supersolvable}. Then the dimension of each simple $U_0(\fg)$-modules is a power of $p$, cf.\ for instance  \cite[(V.8.5)]{SF}. We consider the block $\cB \subseteq U_0(\fg)$ associated to $M$.
Since $\dim_\KK M\!\equiv\!1,-1 \modd(p)$, it follows that $M$ has a one-dimensional composition factor, $\KK_\lambda$ say, for some $\lambda \in X(\fg)$. We conclude that the trivial module $\KK$ is a composition factor
of the indecomposable module $N\!:=\!M\!\otimes_\KK\!\KK_{-\lambda}$, so that $N$ belongs to the principal block $\cB_0(\fg)$ of $U_0(\fg)$. Lemma \ref{Endo1}(2) now provides an indecomposable $U_0(\fg/T(\fg))$-module
$M'$ belonging to the principal block $\cB_0(\fg/T(\fg))$ such that $N \cong \pi^\ast(M')$ is the pull-back of $M'$ along the canonical projection $\pi : U_0(\fg) \lra U_0(\fg/T(\fg))$. In view of Lemma \ref{Endo1}(1), $M'$ has
constant Jordan type $\Jt(M')\!=\!\Jt(N)\!=\!\Jt(M)\!=\![i]\!\oplus\!a_p(M)[p]$ for $i \in \{1,p\!-\!1\}$, and \cite[(5.6)]{CFP08} ensures that $M'$ is endotrivial. According to \cite[(2.3)]{FV00}, the restricted Lie algebra $\fg/T(\fg)$ is
trigonalizable. Thus, (2) shows that $M' \cong \Omega^m_{U_0(\fg/T(\fg))}(\KK_\zeta)$ for some $\zeta \in X(\fg/T(\fg))$ and $m \in \ZZ$. Applying Lemma \ref{Endo1}(2) again, we conclude that the $U_0(\fg)$-module $N$
satisfies
\[ N \cong \Omega^m_{U_0(\fg)}(\KK_{\zeta\circ \pi}).\]
As a result, $M \cong \Omega^m_{U_0(\fg)}(\KK_{(\zeta\circ\pi)+\lambda})$. This concludes the proof of (i). \hfill $\diamond$

\smallskip
\noindent
Let $M$ be indecomposable and endotrivial. We let $K$ be an algebraic closure of $\KK$ and consider the supersolvable restricted Lie algebra $\fg_K\!:=\!\fg\!\otimes_\KK\!K$, whose restricted enveloping
algebra is isomorphic to $U_0(\fg)\!\otimes_\KK\! K$. As $M_K\!:=\!M\!\otimes_\KK\!K$ is an endotrivial $U_0(\fg_K)$-module, (i) now implies that
\[ {\rm (ii)} \ \ \ \ \ \ M_K \cong \Omega_{U_0(\fg_K)}^n(K_\mu)\!\oplus\!({\rm proj.})\]
for some $n \in \ZZ$ and $\mu \in X(\fg_K)$.

Since $\KK$ is perfect and $\fg$ is finite dimensional there exists an intermediate field $\KK \subseteq L \subseteq K$ such that $L\!:\!\KK$ is a finite Galois extension and $\mu(U_0(\fg)\!\otimes 1) \subseteq L$.
Let $\sigma$ be an element of the Galois group $G(L\!:\!\KK)$ of $L\!:\!\KK$. We extend $\sigma$ to an automorphism $\sigma$ of $K$ and note that $\mu\circ (\id_{U_0(\fg)}\otimes \sigma) = \sigma\circ \mu$.  It
follows that
\[ \mu(U_0(\fg)\otimes 1) \subseteq L^{G(L:\KK)}\!=\!\KK.\]
Defining $\lambda \in X(\fg)$ via
\[ \lambda(u) := \mu(u\otimes 1) \ \ \ \ \ \forall \ u \in U_0(\fg)\]
we obtain $K_\mu \cong (\KK_\lambda)_K$. If $(\fg,[p])$ is a torus, then $n\!=\!0$, and (ii) shows that $(\KK_\lambda)_K$ is a direct summand of $M_K$. By the Noether-Deuring Theorem  \cite[(3.8.4(ii))]{EG},
$\KK_\lambda$ is a direct summand of $M$, so that $M$ being indecomposable yields $M\cong \KK_\lambda$. Alternatively, we obtain $M_K\! \oplus\!({\rm proj.}) \cong (\Omega^n_{U_0(\fg)}(\KK_\lambda))_K\!\oplus
\!({\rm proj.})$, with $\KK_\lambda$ being non-projective.\footnote{If $\KK_\lambda$ is projective, so is the trivial module $\KK$ and the Hopf algebra $U_0(\fg)$ is semisimple. As $\KK$ is perfect, this implies that $(\fg,[p])$
is a torus, see \cite[Thm.]{Ho54} and \cite[(2.3.7)]{SF}.} Hence there is a projective $U_0(\fg_K)$-module $Q$, and a projective $U_0(\fg)$-module $P$ such that
\[ (M\!\oplus\! P)_K \cong M_K\!\oplus\! P_K \cong \Omega^n_{U_0(\fg)}(\KK_\lambda)_K\!\oplus\!Q.\]
The Noether-Deuring Theorem  \cite[(3.8.4(ii))]{EG} implies that $\Omega^n_{U_0(\fg)}(\KK_\lambda)$ is a direct summand of $M\!\oplus\!P$. Since $M$ and $\Omega^n_{U_0(\fg)}(\KK_\lambda)$ are non-projective
indecomposable, the Theorem of Krull-Remak-Schmidt yields $M \cong \Omega^n_{U_0(\fg)}(\KK_\lambda)$. This implies the desired result. \hfill $\square$

\bigskip

\begin{Remarks} (1) If $\fu$ is unipotent, then the proof above works for every field $\KK$.

(2) Suppose that $\dim V(\fg_K)\!\ge\!2$. Then $n$ and $\lambda$ are uniquely determined. As a result, the subgroup of the auto-equivalences of the stable category $\underline{\modd}\,U_0(\fg)$ defined by endotrivial
modules is isomorphic to $\ZZ\!\oplus\! X(\fg)$ (so that $X(\fg)$ is the torsion group). \end{Remarks}

\bigskip

\begin{Corollary} \label{Endo2} Let $\KK$ be algebraically closed and suppose that $G$ is a connected, solvable algebraic group. If $M \in \modd U_0(\Lie(G))$ is endotrivial, then there exists $\lambda
\in X(\Lie(G))$ and $n \in \ZZ$ such that
\[ M \cong \Omega^n_{U_0(\Lie(G))}(\KK_\lambda)\!\oplus\!({\rm proj.}).\] \end{Corollary}

\begin{proof} By the Lie-Kolchin Theorem, $\Lie(G)$ is trigonalizable. Now apply Theorem B. \end{proof}

\bigskip
\noindent
In general, syzygy functions are not necessarily constant. Letting $\{e,h,f\}$ be the standard basis of $\mathfrak{sl}(2)$, we consider the central extension $\mathfrak{sl}(2)_s\!:=\!\mathfrak{sl}(2)\!\oplus\!\KK c_0$
of $\mathfrak{sl}(2)$ by the one-dimensional elementary abelian ideal $\KK c_0$. By definition, the extension splits as an extension of ordinary Lie algebras.
The $p$-map is given by
\[(x,\alpha c_0)^{[p]}=(x^{[p]},\psi(x)c_0),\]
where the $p$-semilinear map $\psi:\mathfrak{sl}(2)\rightarrow \KK$ satisfies $\psi(e)=0=\psi(f)$ and $\psi(h)=1$. This readily implies
\[ V(\fsl(2)_s) = \KK e\!\oplus\!\KK c_0 \cup \KK f\!\oplus\!\KK c_0 \ \ \text{as well as} \ \  \EE(2,\fsl(2)_s) = \{ \fe_e, \fe_f\},\]
where $\fe_x\!:=\!\KK x\!\oplus\!\KK c_0$ for $x \in \{e,f\}$. We put $\fb_s\!:=\!\KK (h\!+\!c_0)\!\oplus\! \KK e$ and consider the {\it baby Verma module}
\[Z(0) :=U_0(\mathfrak{sl}(2)_s)\!\otimes_{U_0(\fb_s)}\!\KK\]
with $h\!+\!c_0$ acting trivially on $\KK$.

Let $\Rad(Z(0))$ be the radical of $Z(0)$. In view of \cite[(4.3.4)(1)]{CF21} and \cite[(5.6)]{CFP08}, the indecomposable module $\Rad(Z(0))$ is endotrivial. Moreover, we have
$s_{\Rad(Z(0))}(\fe_f)\!=\!1,~s_{\Rad(Z(0))}(\fe_e)\!=\!-1$. Since $\Rad(Z(0))$ is endotrivial, \cite[(7.2)]{FS02} shows that the stable Auslander-Reiten component containing $\Rad(Z(0))$
is of type $\ZZ[A^\infty_\infty]$. A consecutive application of \cite[(3.1.2)]{Fa14} and \cite[(5.6)]{CFP08} now implies that every module belonging to this component is endotrivial. In contrast to Theorem B, the class of
indecomposable endotrivial $U_0(\fsl(2)_s)$-modules is not the union of finitely many $\Omega$-orbits.

\bigskip

\section{Modules with one non-projective block}
Let $\KK$ be an algebraically closed field of characteristic $p\!\ge\!3$, $(\fg,[p])$ be a restricted Lie algebra over $\KK$. Recall that $\CJT(\fg) \subseteq \modd U_0(\fg)$ denotes the full subcategory of modules of
constant Jordan type.

We begin with the following direct consequence of \cite{Be10}:

\bigskip

\begin{Lemma} \label{NP1} Let $M \in \CJT(\fg)$ be of constant Jordan type $\Jt(M)\!=\![i]\!\oplus\!a_p(M)[p]$ for some $i \in \{1,\ldots,p\!-\!1\}$. If $\EE(2,\fg)\!\ne\!\emptyset$, then $i \in \{1,p\!-\!1\}$ and $M$ is
endotrivial. \end{Lemma}

\begin{proof} Let $\fe \in \EE(2,\fg)$. Then $N\!:=\!M|_{U_0(\fe)}$ has constant Jordan type $\Jt(N)\!=\!\Jt(M)$. Let $\alpha : \KK[T]/(T^p) \lra U_0(\fe)$ be a homomorphism of associative $\KK$-algebras such that
the pull-back $\alpha^\ast(U_0(\fe)) \in \modd \KK[T]/(T^p)$ is projective. Owing to \cite[(3.8)]{FP05}, we have $\alpha^\ast(N) \cong [i]\!\oplus\! a_p(M)[p]$.

The associative algebra $U_0(\fe)$ is isomorphic to the group algebra $\KK E_2$ of the elementary abelian $p$-group $E_2\!=\!\ZZ/(p)^2$ of rank $2$. Hence there is $N' \in \modd \KK E_2$ such that
$\beta^\ast(N') \cong [i]\!\oplus\!a_p(M)[p]$ for every $\beta : \KK[T]/(T^p) \lra \KK E_2$ with $\beta^\ast(\KK E_2)$ being projective. Thus, $N'$ has constant Jordan type $\Jt(N')\!=\!\Jt(M)$ and \cite[(1.1)]{Be10} yields
$i \in \{1,p\!-\!1\}$. The second assertion follows from \cite[(5.6)]{CFP08}. \end{proof}

\bigskip
\noindent
We denote by $\fb_{\fsl(2)}$ and $\fu_{\fsl(2)}$ the Borel subalgebra of upper triangular matrices of $\fsl(2)$ and its unipotent radical, respectively. We say that $\fg$ is {\it generically toral}, provided every torus
$\ft \subseteq \fg$ of maximal dimension is self-centralizing (see \cite[\S1]{CF19} for more details).

\bigskip

\begin{Theorem} \label{NP2} Let $(\fg,[p])$ be a restricted Lie algebra such that $V(\fg)\!\ne\!\{0\}$. If $M \in \CJT(\fg)$ is indecomposable of constant Jordan type $\Jt(M)\!=\![i]\!\oplus\!a_p(M)[p]$ for some
$i \in \{2,\ldots,p\!-\!2\}$, then one of the following statements holds:
\begin{enumerate}
\item[(a)] $\fg/C(\fg) \cong \fsl(2)$, with $C(\fg)$ being a torus, or
\item[(b)] $\fg/T(\fg) \cong \fb_{\fsl(2)}, \fu_{\fsl(2)}$. \end{enumerate} \end{Theorem}

\begin{proof} We necessarily have $p\!\ge\!5$, and it follows from Lemma \ref{NP1} that $\EE(2,\fg)\!=\!\emptyset$. As $V(\fg)\!\ne\!\{0\}$, the restricted Lie algebra $(\fg,[p])$ has $p$-rank $\rk_p(\fg)\!=\!1$, cf.\ \cite{CF19}.
Hence \cite[(4.2.4)]{CF19} yields:

\smallskip

($\star$) \ \textit{If $\fg$ is generically toral, then $\fg/C(\fg) \cong \fsl(2), \fb_{\fsl(2)}^{-1}, \fb_{\fsl(2)}$, where}
\[ \fb_{\fsl(2)}^{-1}\!:=\!\KK t\!\oplus\!\KK x\!\oplus\! \KK y \ \ ; \ \ [t,x]\!=\!x\ , \  [t,y]\!=\!-y \ , \  [x,y]\!=\!0 \ ; \ t^{[p]}\!=\!t \ , \ x^{[p]}\!=\!0\!=\!y^{[p]}.\]

\smallskip
\noindent
Let $\ft \subseteq\fg$ be a torus of maximal dimension, $r(\fg)$ be the number of roots of $\fg$ relative to $\ft$.  We proceed in several steps.

\smallskip
\noindent
(i) {\it Suppose that $\dim V(\fg)\!\ge\!2$}.

If $V(C(\fg))\!\ne\!\{0\}$, then $\EE(2,\fg)\!\ne\!\emptyset$, a contradiction. Consequently, $C(\fg)\!=\!T(\fg)$ is a torus.

If $r(\fg)\!\ge\!2$, then \cite[(3.1)]{CF19} ensures that $\fg$ is generically toral and ($\star$) yields $\fg/C(\fg) \cong \fsl(2), \fb_{\fsl(2)}^{-1}$. If $\fg/C(\fg)\cong  \fb_{\fsl(2)}^{-1}$, then $\fg$ is solvable and \cite[(5.8.5)]{SF}
shows that every simple $U_0(\fg)$-module has dimension a power of $p$, while $\dim_\KK M\!\equiv\! i \ \modd(p)$. Consequently, $M$ possesses a one-dimensional composition factor $\KK_\lambda$ for some
$\lambda \in X(\fg)$. We conclude that $N\!:=\!M\!\otimes_\KK\!\KK_{-\lambda}$ is indecomposable and has constant Jordan type $\Jt(N)\!=\!\Jt(M)$. Since the trivial module $\KK$ is a composition factor of $N$,
$N$ belongs to the principal block $\cB_0(\fg)$.  It now follows from Lemma \ref{Endo1} that there is an indecomposable $U_0(\fg/C(\fg))$-module $M'$ of constant Jordan type $\Jt(M')\!=\!\Jt(M)$. Since
$\EE(2,\fb_{\fsl(2)}^{-1})\!\ne\!\emptyset$, this contradicts Lemma \ref{NP1}. We obtain $\fg/C(\fg) \cong \fsl(2)$, so that (a) holds.

If $r(\fg)\!=\!1$, there is a root $\alpha \in \ft^\ast\!\smallsetminus\!\{0\}$ such that
\[ \fg = \fg_0\!\oplus \fg_\alpha\]
with $\fg_0$ being a Cartan subalgebra of $\fg$. Thus, $\fg_\alpha$ is an abelian ideal of $\fg$, whence $\fg_\alpha^{[p]} \subseteq C(\fg) \subseteq \ft$. It follows that $\fs\!:=\!\ft\!\oplus\!\fg_{\alpha}$ is a $p$-subalgebra of
$\fg$ such that $\rk_p(\fs)\!=\!1$. In view of $\fs$ being generically toral, ($\star$) yields $\fs/C(\fs) \cong \fb_{\fsl(2)}$, so that $\dim_\KK \fg_\alpha\!=\!1$. Let $\fg_\alpha\!=\!\KK x$.
Since $x^{[p]}$ belongs to the toral $p$-subalgebra $C(\fg)$, there is $c \in C(\fg) \subseteq \fg_0$ such that $x\!-\!c \in V(\fg)$. Let $y \in V(\fg_0)$. Then $y$ acts nilpotently on $\fg_\alpha$, whence $[y,x|\!=\!0$.
Consequently, the $p$-subalgebra $\fe\!:=\!\KK(x\!-\!c)\!\oplus\!\KK y$ is elementary abelian. The $U_0(\fs)$-module $M|_{U_0(\fs)}$ has constant Jordan type $\Jt(M|_{U_0(\fs)})\!=\!\Jt(M)$ and thus possesses an
indecomposable constituent $N$ of constant Jordan type $\Jt(N)\!=\![i]\!\oplus\!a_p(N)[p]$. Lemma \ref{NP1} yields $\EE(2,\fs)\!=\!\emptyset$, so that $y\!=\!0$. It follows that $V(\fg_0)\!=\!\{0\}$, so that $\fg_0\!=\!\ft$ is a
torus, cf.\ \cite[(2.3.7),(2.3.10)]{SF}. As a result, $\fg$ is generically toral and ($\star$) yields $\fg/C(\fg)\cong\!\fb_{\fsl(2)}$. Now Lemma \ref{Endo1}(1) implies $\dim V(\fg)\!\le\!1$, a contradiction.

We finally assume that $r(\fg)\!=\!0$, so that $\fg$ is nilpotent with toral center $C(\fg)$. Hence $\fu\!:=\!\fg/C(\fg)$ is unipotent and Lemma \ref{Endo1}(1) ensures that $\dim V(\fu)\!\ge\!2$. Now $\fu$ being unipotent entails
$\EE(2,\fu)\!\ne\!\emptyset$. As above, there exists a $U_0(\fu)$-module $M'$ of constant Jordan type $\Jt(M')\!=\!\Jt(M)$, which contradicts Lemma \ref{NP1}. \hfill $\diamond$

\smallskip
\noindent
(ii) {\it Suppose that $\dim V(\fg)\!=\!1$}.

By \cite[(4.3)]{Fa95}, we have
\[ \fg/T(\fg)\cong \KK t\!\ltimes\!(\KK x)_p\]
for some toral element $t \in \fg/T(\fg)$ (i.e., $t^{[p]}\!=\!t)$, and some $p$-nilpotent element $x \in \fg/T(\fg)$. We first assume that $T(\fg)\!=\!(0)$. Then $\fg$ is trigonalizable, and as $x$ is $p$-nilpotent and $V(\fg)\!\ne\!\{0\}$,
there is $n \in \NN_0$ such that $x^{[p]^n}\!\ne\!0\!=\!x^{[p]^{n+1}}$. Writing $\fn\!:=\!(\KK x)_p$, we observe that the two-sided ideal $J\!:=\!U_0(\fg)\fn$ is nilpotent, while $U_0(\fg)/J\cong U_0(\KK \ft)$ is semisimple. Hence
$J\!=\! U_0(\fg)x\!=\!x U_0(\fg)$ is the Jacobson radical of $U_0(\fg)$.

Let $M \in \modd U_0(\fg)$ be indecomposable. Since $U_0(\fg)$ is a Nakayama algebra (cf.\ \cite[(3.2)]{Fa95}), there is a principal indecomposable module $P$ and $\ell \in \{1,\ldots, p^{n+1}\}$ such that $M \cong P/J^\ell
P$, cf.\ \cite[(5.3.5)]{ASS06}. Moreover, there exists $\lambda \in X(\fg)$ such that $P\cong U_0(\fg)\!\otimes_{U_0(\KK t)}\!\KK_\lambda$, see for instance \cite[Prop.1]{Fe99}. Consequently,
\[ M|_{U_0(\fn)} \cong U_0(\fn)/x^\ell U_0(\fn).\]
We write $\ell\!=\!sp^n\!+\!j$ for $j \in \{0,\ldots, p^n\!-\!1\}$ and $s \in \{0,\ldots,p\}$. Setting $y\!:=\!x^{[p]^n}$, we have $V(\fg)\!=\!V(\fn)\!=\!\KK y$. If $s\!=\!0$, then $y.M\!=\!(0)$, so that $M|_{U_0(\KK y)} \cong \ell [1]$. We
assume that $s\!\ge\!1$, whence $\ell\!\ge\!p^n$. For every  $q \in \{0,\ldots, p^n\!-\!1\}$, we put $I_q\!:=\!\{ i \in \{0,\ldots, \ell\!-\!1\} \   ;   \   i \! \equiv \! q \ \modd (p^n)\}$. Let $z\!:=x\!+\!x^\ell U_0(\fn) \in M$.  The space
\[ M_q \!:= \!\bigoplus_{i \in I_q} \KK z^i\]
is a cyclic $U_0(\KK y)$-submodule of $M$ (with generator $z^q$) such that $M_q\cong [s]$ for $q\!=\!0$ and $q \in \{j\!+\!1, \ldots, p^n\!-\!1\}$, while $M_q\cong [s\!+\!1]$ for $q \in \{1,\ldots, j\}$ (in that case, we necessarily
have $s\!\le\!p\!-\!1$). Consequently,
\[ \Jt(M) = \Jt(M|_{U_0(\fn)}) = [M|_{U_0(\KK y)}] = [\bigoplus_{q=0}^{p^n-1}M_q] = (p^n\!-\!j)[s]\!\oplus\!j[s\!+\!1].\]
Accordingly, we have $\Jt(M)\!=\![i]\!\oplus\!a_p(M)[p]$ for some $2\!\le\!i\!\le\!p\!-\!2$ only if $n\!=\!0$. Consequently, $\fg \cong \fb_{\fsl(2)}$ or $\fg \cong \KK y \cong \fu_{\fsl(2)}$.

In the general case, we consider $\fg'\!:=\!\fg/T(\fg)$. If $M \in \modd U_0(\fg)$ is indecomposable of constant Jordan type $\Jt(M)\!=\![i]\!\oplus\!a_p(M)[p]$ for $2\!\le\!i\!\le\!p\!-\!2$, Lemma \ref{Endo1}
provides an indecomposable module $M' \in \CJT(\fg')$ such that $\Jt(M')\!=\!\Jt(M)$. By what we have seen, we obtain $\fg' \cong \fb_{\fsl(2)}, \fu_{\fsl(2)}$. \hfill $\diamond$

Since $V(\fg)\!\ne\!\{0\}$ is conical, our result follows from (i) and (ii).  \end{proof}

\bigskip

\begin{Remarks} (1) Suppose that $\fg/T(\fg) \cong \fb_{\fsl(2)},\fu_{\fsl(2)}$ and let $\ft$ be a maximal torus of $\fg$. Given $\lambda \in X(\fg/T(\fg))$ and $i \in \{1,\ldots, p\!-\!1\}$, we consider the principal indecomposable
module $P(\lambda)\!:=\!U_0(\fg)\!\otimes_{U_0(\ft)}\!\KK_{\lambda\circ \pi}$ as well as $M_i(\lambda)\!:=\!P(\lambda)/J^iP(\lambda)$. The above proof then shows that $\{M_i(\lambda) \ ; \ \lambda \in X(\fg/T(\fg))\}$ is the set
of isoclasses of the indecomposable $U_0(\fg)$-modules of constant Jordan type with $[i]$ as only non-projective block.

(2) Suppose that $\fg/C(\fg) \cong \fsl(2)$, with $C(\fg)$ being a torus. We let $V(d) \in \modd U_0(\fsl(2))$ be the Weyl module with highest weight $d \!\not \equiv\! -1 \modd (p)$. If $M \in \CJT(\fg)$ is
indecomposable of constant Jordan type with only non-projective block $[i]$, then $\dim_\KK M\!=\!sp\!+\!i$, and there exists a character $\lambda \in X(\fg)$ such that
$M \cong \pi^\ast(V(\dim_\KK M\!-\!1))\!\otimes_\KK \!\KK_\lambda, \pi^\ast(V(\dim_\KK M\!-\!1)^\ast)\!\otimes_\KK \!\KK_\lambda$ is the pull-back of a Weyl module or its dual along the canonical projection
$\pi : U_0(\fg) \lra U_0(\fg/C(\fg))$.

(3) Let $\cG$ be a solvable infinitesimal group scheme with maximal multiplicative normal subgroup $\cM(\cG)$. Suppose there is a $\cG$-module $M$ of constant Jordan type $\Jt(M)\!=\![i]\!\oplus\!a_p(M)[p]$ for some
$i \in \{2,\ldots,p\!-\!2\}$. In view of \cite[\S1]{SFB97}, \cite[\S1,\S2]{FV99} and \cite[\S5]{CF19}, one can show that $\cG/\cM(\cG) \cong \GG_{a(1)}\!\rtimes\!\GG_{m(s)}$ for some $s\!\ge\!0$. Moreover, in each of these cases,
such modules do exist. The determination of the endotrivial modules, however, would necessitate the extension of our geometric techniques to the context of unipotent infinitesimal group schemes. \end{Remarks}

\bigskip

\bigskip

\begin{center}
\textbf{Acknowledgement}
\end{center}
The first named author would like to thank Professor Jon Carlson for helpful discussions.

\end{document}